\documentclass[reqno]{amsart}

\usepackage{graphicx,amssymb,amsmath,lineno,eucal}

\newtheorem{Theorem}{Theorem}
\newtheorem{Corollary}{Corollary}
\newtheorem{Lemma}{Lemma}
\theoremstyle{definition}
\newtheorem{Definition}{Definition}
\theoremstyle{remark}

\title[Characterization theorem for Laguerre-Hahn OP on snul] {Characterization theorem for
Laguerre-Hahn orthogonal polynomials on non-uniform lattices}

\author[A. Branquinho]{\sc A. Branquinho}
\address[A. Branquinho]{CMUC and Department of Ma\-the\-ma\-tics, University of Coimbra,
Apar\-ta\-do 3008, EC Santa Cruz, 3001-501 COIMBRA, Portugal.}
\email[A. Branquinho]{ajplb@mat.uc.pt}

 \author[M.N.  Rebocho]{\sc M.N. Rebocho$^{*}$}
 \address[M.N. Rebocho]{Departamento de Matem\'atica,
 Universidade da Beira Interior, 6201-001 Covilh\~a, Portugal;
 CMUC, University of Coimbra, Apartado 3008, EC Santa Cruz,
3001-501 Coimbra, Portugal.}
 \email[M.N. Rebocho]{mneves@ubi.pt}

\thanks{$*$ Corresponding author.
}

\begin{document}

 \begin{abstract}
It is  stated and proved a characterization theorem for
Laguerre-Hahn orthogonal polynomials on non-uniform lattices.  This
theorem proves the equivalence between the Riccati equation for the
formal Stieltjes function, linear first-order  difference relations
for the orthogonal polynomials  as well as for the associated
polynomials of the first kind, and linear first-order difference
relations for the  functions of the second kind.
 \end{abstract}

 \maketitle

\section{Introduction}\label{sec:1}

The present paper concerns orthogonal polynomials of a discrete
variable on
 non-uniform lattices (commonly
 denoted by snul). These lattices
are associated with  divided differences operators such as the
Wilson or Askey-Wilson operator (\cite[Section 5]{askey-wilson}, and
\cite{atak-etal,hahn-q,niki-snul,niki-sus-uv}). Specifically, we
focus our attention on the so-called Laguerre-Hahn orthogonal
polynomials. The Laguerre-Hahn orthogonal polynomials on non-uniform
lattices  were introduced by A. Magnus in \cite{magnus-LHsnul}, as
the ones for which the formal Stieltjes function  satisfies a
Riccati difference equation with polynomial coefficients, with the
difference operator taken as a  general divided difference operator
given by \cite[Eq. (1.1)]{magnus-LHsnul} (see Section \ref{sec:2} of
the present paper for the precise definitions and main properties).
In this pioneering work, Magnus establishes difference relations as
well as representations for the Laguerre-Hahn orthogonal polynomials
and he proves that, under certain restrictions on the degrees of the
coefficient of the Riccati difference equation, the Laguerre-Hahn
orthogonal polynomials are the associated Askey-Wilson polynomials
\cite{andrews-askey,askey-wilson}.

As it is well known from the setting of continuous orthogonality,
 Laguerre-Hahn
orthogonal polynomials inherit many properties from the classical
and semi-classical families
\cite{sylvesterOPRL,dini,magnus-ric,maroni}. Indeed, one of the
research topics within the Laguerre-Hahn theory of a discrete
variable is the  so-called structure relations, that is, linear
difference relations involving the orthogonal polynomials (see
 \cite{ban,fou-diffeq,Mamma-koef-ronv,fou-paco} and their lists of
 references).
In the semi-classical case, it was proven in \cite{magnus-snul-sc}
the characterization of semi-classical orthogonal polynomials on
non-uniform lattices in terms of  structure relations. A more recent
contribution, \cite{fou-ITSF}, proves the characterization of
classical polynomials on non-uniform lattices in terms of two types
of structure relations, using the so-called functional approach.

In the present paper we show a characterization theorem for
Laguerre-Hahn orthogonal polynomial on arbitrary non-uniform
lattices. Our main result is given in Theorem \ref{teo2}, where it
is shown the equivalence between:

(a) the Riccati difference equation for the formal Stieltjes
function, $S$;

\hangindent=1.2cm \hangafter=1 (b) linear first-order  difference
relations for orthogonal polynomials related to $S$, as well as for
the associated polynomials of the first kind;

\hangindent=1.2cm \hangafter=1 (c) linear first-order difference
relations for the functions of the second kind related to $S$.

The difference relations contained in Theorem \ref{teo2} for
Laguerre-Hahn families extend some of the difference relations for
the classical families given in
\cite{fou-ITSF,magnus-snul-sc,witte}.

 This paper is organized as follows. In Section~\ref{sec:2} we give the definitions and state the basic results which
 will be used in the forthcoming sections. In Section~\ref{sec:3}  we show the main results of the paper,  namely,
 the equivalence between the above referred conditions
 (a), (b) and (c), stated in Theorem \ref{teo2}
 Section~\ref{sec:4} is devoted to the proof of Theorem \ref{teo2}.

\section{Preliminary results}\label{sec:2}

\subsection{The operators $\mathbb{D}, \mathbb{E}_j, \mathbb{M}$ and the related non-uniform lattices}

We consider the divided difference operator $\mathbb{D}$ given in
 \cite[Eq.
(1.1)]{magnus-LHsnul}, involving the values of a function at two
points, with the fundamental property that $\mathbb{D}$ leaves a
polynomial of degree $n-1$ when applied to a polynomial of degree
$n$.  The operator $\mathbb{D}$, defined on the space of arbitrary
functions, is given by
\begin{equation}
(\mathbb{D}f)(x)=\frac{f(y_2(x))-f(y_1(x))}{y_2(x)-y_1(x)}\,,\label{eq:def-D}
\end{equation}
where, at this stage, $y_1$ and $y_2$ are still unknown functions.
To define them, one starts by using the property that  $\mathbb{D}f$
is a polynomial of degree $n-1$ whenever $f$ is a polynomial of
degree $n$. Then, applying $\mathbb{D}$ to $f(x)=x^2$ and
$f(x)=x^3$,  one obtains, respectively,
\begin{gather}
y_1(x)+y_2(x)= \textrm{polynomial of degree } 1\,, \label{eq:cond-y1}\\
(y_1(x))^2+y_1(x)y_2(x)+(y_2(x))^2= \textrm{polynomial of degree }
2\,, \label{eq:cond-y2}
\end{gather}
 the later condition being equivalent to $y_1(x)y_2(x)=$
polynomial of degree less or equal than $2$. The  conditions
(\ref{eq:cond-y1})-(\ref{eq:cond-y2}) define $y_1$ and $y_2$ as the
two roots of a quadratic equation
\begin{equation}
\hat{a}y^2+2\hat{b}xy+\hat{c}x^2+2\hat{d}y+2\hat{e}x+\hat{f}=0\,,
\quad \hat{a}\neq 0\,.\label{eq:conica}
\end{equation}

Identities involving $y_1$ and $y_2$, following from the fact that
$y_1, y_2$ are the roots of~(\ref{eq:conica}):
\begin{eqnarray}
&&y_1(x)+y_2(x)=-2(\hat{b}x+\hat{d})/\hat{a}\,,\label{eq:somay1-y2}\\
&&y_1(x)y_2(x)=(\hat{c}x^2+2\hat{e}x+\hat{f})/\hat{a}\,,\label{eq:prody1-y2}\\
&&y_1(x)=p(x)-\sqrt{r(x)}\,,
\;\;y_2(x)=p(x)+\sqrt{r(x)}\,,\label{eq:y1-y2}
\end{eqnarray}
with $p, r$ polynomials given by
$p(x)=-\displaystyle\frac{\hat{b}x+\hat{d}}{\hat{a}}\,,\;
r(x)=\displaystyle\frac{\lambda}{\hat{a}^2}
\left(x+\frac{\hat{b}\hat{d}-\hat{a}\hat{e}}{\lambda}\right)^2+\frac{\tau}{\hat{a}\lambda}\,,\;$
where $\lambda=\hat{b}^2-\hat{a}\hat{c},\;
 \tau=\left((\hat{b}^2-\hat{a}\hat{c})(\hat{d}^2-\hat{a}\hat{f})-(\hat{b}\hat{d}-\hat{a}\hat{e})^2\right)/\hat{a}.$

 There are four primary classes of lattices and related divided difference operators~(\ref{eq:def-D}):\\
(i) the linear lattice, related to the forward difference operator
\cite[Chapter 2, Section~12]{nik-MIR}\,;\\
(ii) the $q$-linear lattice, related to the $q$-difference operator
 \cite{hahn-q}\,;\\
(iii) the quadratic lattice, related to the Wilson operator
\cite{askey-wilson}\,;\\
(iv) the $q$-quadratic lattice, related to the Askey-Wilson operator
 \cite{askey-wilson}.\\
Such a classification  of lattices is done according to the two
parameters $\lambda$ and $\tau$ above defined, assuming
                                       $\hat{a}\hat{c}\neq 0$:  $\lambda=\tau=0$ in case (i); $\lambda\neq 0, \tau=0$ in case (ii);
                                       $\lambda=0, \tau \neq 0$ in case (iii); $\lambda\,\tau  \neq 0$ in case (iv).
 Each of the operators in (i)-(iv) is an extension of the preceding
one, which is recovered as a particular case or as a limit case, up
to a linear transformation of the variable.

In \cite[Section 2]{magnus-snul-sc} it is given a geometric
interpretation of the lattices. For the quadratic class of lattices
(the so-called  snul), it is possible to have a parametric
representation of the conic (\ref{eq:conica}), say $\left\{x(s),
y(s)\right\}$, such that $y_1(x(s))=y(s)=x(s-1/2)$ and
$y_2(x(s))=y(s+1)=x(s+1/2)$, thus leading to
\cite{atak-etal,niki-snul,niki-sus-uv} $x(s)=c_1q^s+c_2q^{-s}+c_3,$
$q+q^{-1}=4\hat{b}^2/(\hat{a}\hat{c})-2,$ if $\lambda\,\tau \neq 0,$
and $x(s)=c_4s^2+c_5s+c_6,$ if $\lambda=0, \tau  \neq 0$, with
appropriate constants $c$'s.

 In the present paper we
will consider the general case $\lambda\,\tau \neq 0.$ Throughout
the paper we shall use the notation $\Delta_y=y_2-y_1.$ Note that
$\Delta_y=2\sqrt{r}.$ We shall operate with the divided difference
operator given in its general form (\ref{eq:def-D}). By defining the
operators $\mathbb{E}_1$ and $\mathbb{E}_2$ (see
\cite{magnus-LHsnul}), acting on arbitrary functions $f$ as
\begin{equation}
(\mathbb{E}_1f)(x)=f(y_1(x))\,, \;\;(\mathbb{E}_2f)(x)=f(y_2(x))\,,
\label{eq:E1-E2-defi}\nonumber
\end{equation}
 then (\ref{eq:def-D}) is given by
\begin{equation}
(\mathbb{D}f)(x)=\frac{(\mathbb{E}_2f)(x)-(\mathbb{E}_1f)(x)}{y_2(x)-y_1(x)}\,.\nonumber
\end{equation}

 The companion operator of $\mathbb{D}$ is defined  as (see
 \cite{magnus-LHsnul})
\begin{equation}
(\mathbb{M}f)(x)=\frac{(\mathbb{E}_1f)(x)+(\mathbb{E}_1f)(x)}{2}\,.\label{eq:def-M}
\nonumber
\end{equation}
Some useful identities involving $\mathbb{D}, \mathbb{M}$ and
$\mathbb{E}_1, \mathbb{E}_2$ are listed below (see
\cite{magnus-LHsnul}):
\begin{eqnarray}
&&\mathbb{D}(gf)=\mathbb{D}g\,\mathbb{M}f+\mathbb{M}g\,\mathbb{D}f\,,\label{eq:D-prod}\\
&&\mathbb{D}(g/f)=\frac{\mathbb{D}g\,\mathbb{M}f-\mathbb{D}f\,\mathbb{M}g}{\mathbb{E}_1f\, \mathbb{E}_2f}\,,\label{eq:D-quoc}\\
&&\mathbb{D}(1/f)=\frac{-\mathbb{D}f}{\mathbb{E}_1f\, \mathbb{E}_2f}\,,\label{eq:D-1sobref}\\
&&\mathbb{M}(gf)=\mathbb{M}g\,  \mathbb{M}f+\frac{\Delta_y^2}{4}\,\mathbb{D}g\,\mathbb{D}f \,,\label{eq:M-prod}\\
&&\mathbb{M}(g/f)=\frac{\mathbb{E}_1g\,
\mathbb{E}_2f+\mathbb{E}_2g\,\mathbb{E}_1f}{2\mathbb{E}_1f\, \mathbb{E}_2f}\,,\label{eq:M-quoc} \\
&&\mathbb{M}(1/f)=\frac{\mathbb{M}f}{\mathbb{E}_1f\,
\mathbb{E}_2f}\,.\label{eq:M-1sobref}
\end{eqnarray}
Eq. (\ref{eq:D-prod}) has the equivalent forms: \begin{eqnarray}
&&\mathbb{D}(gf)=\mathbb{D}g\,
\mathbb{E}_1f+\mathbb{D}f\,\mathbb{E}_2g\,,\label{eq:Dprod-gf-E1}\\
&&\mathbb{D}(gf)=\mathbb{D}g\,
\mathbb{E}_2f+\mathbb{D}f\,\mathbb{E}_1g\,.\label{eq:Dprod-gf-E2}
\end{eqnarray}
Also, one has two equivalent forms for (\ref{eq:D-quoc}):
\begin{eqnarray}
&&\mathbb{D}(g/f)=\frac{\mathbb{D}g\,
\mathbb{E}_1f-\mathbb{D}f\,\mathbb{E}_1g}{\mathbb{E}_1f\, \mathbb{E}_2f}\,,\label{eq:D-quoc-E1}\\
&&\mathbb{D}(g/f)=\frac{\mathbb{D}g\mathbb{E}_2f-\mathbb{D}f\mathbb{E}_2g}{\mathbb{E}_1f\,
\mathbb{E}_2f}\,. \label{eq:D-quoc-E2}
\end{eqnarray}

Further identities to be used throughout the paper are given in the
 following lemma.
\begin{Lemma}\label{eq:lemma-ident-aux} The following assertions hold:\\
(a)\; $\mathbb{E}_1f\,\mathbb{E}_2f$ is a polynomial whenever $f$ is
a polynomial;\\
(b) \;$\mathbb{M}f$ is a polynomial whenever $f$ is
a polynomial;\\
(c)\; $(\mathbb{E}_1f)^2+(\mathbb{E}_2f)^2$ is a polynomial whenever
$f$ is a polynomial;\\
 (d)\; for arbitrary functions $f, g$, the following
identities take place: \begin{eqnarray}
&&\mathbb{E}_1f\,\mathbb{E}_2g+\mathbb{E}_1g\,\mathbb{E}_2f
=2\mathbb{M}g\,\mathbb{M}f-\frac{\Delta_y^2}{2}\,\mathbb{D}g\,\mathbb{D}f\,, \label{eq:identE1E2}\\
&&-\mathbb{D}f\,\mathbb{E}_1g+\mathbb{D}g\,\mathbb{E}_1f=2\mathbb{M}f\,\mathbb{D}g-\mathbb{D}(gf)\,.
\label{eq:ident-D-DE1E2}
\end{eqnarray}
 Therefore, $\mathbb{E}_1f\,\mathbb{E}_2g+
\mathbb{E}_1g\,\mathbb{E}_2f$ and
$-\mathbb{D}f\,\mathbb{E}_1g+\mathbb{D}g\,\mathbb{E}_1f$ are
polynomials whenever $f$ and $g$ are polynomials.
\end{Lemma}
\begin{proof}
 Let $f$ be an arbitrary polynomial, and let us write
$f(x)=\prod_{k=0}^{n}(x-x_{n,k})$.
\begin{eqnarray*}
(\mathbb{E}_1f\,\mathbb{E}_2f)(x)&=&\prod_{k=0}^{n}(y_1(x)-x_{n,k})\prod_{k=0}^{n}(y_2(x)-x_{n,k})\\
&=&\prod_{k=0}^{n}\left(y_1(x)y_2(x)-(y_1(x)+y_2(x))x_{n,k}+x^2_{n,k}\right)\,.
\end{eqnarray*}
From (\ref{eq:somay1-y2}) and (\ref{eq:prody1-y2}) there follows
$$(\mathbb{E}_1f\,\mathbb{E}_2f)(x)=\prod_{k=0}^{n}\left(\frac{\hat{c}x^2+2\hat{e}x+\hat{f}}{\hat{a}}+\frac{2}{\hat{a}}(\hat{b}x+\hat{d})x
+x^2_{n,k}\right)\,,$$ which is a polynomial, thus proving assertion
(a).

Let us now write $f$ as $f(x)=\sum_{k=0}^{n}a_{n,k}x^k$. Recall
(\ref{eq:y1-y2}), $y_1(x)=p(x)-\sqrt{r(x)},
y_2(x)=p(x)+\sqrt{r(x)}$. Thus,
\begin{eqnarray*}
(\mathbb{E}_1f+\mathbb{E}_2f)(x)&=&\sum_{k=0}^{n}a_{n,k}((y_1(x))^k+(y_2(x))^k)\\
&=&\sum_{k=0}^{n}a_{n,k}\left((p(x)-\sqrt{r(x)})^k+(p(x)+\sqrt{r(x)})^k\right)\,.
\end{eqnarray*}
 As $\displaystyle (p\pm \sqrt{r})^k=\sum_{j=0}^{k}\left(
                                          \begin{array}{c}
                                            k \\
                                            j \\
                                          \end{array}
                                        \right)
p^j\left(\pm \sqrt{r}\right)^{k-j}\,,$ we obtain
$$(\mathbb{E}_1f+\mathbb{E}_2f)(x)=\sum_{k=0}^{n}a_{n,k}\sum_{j=0}^{k}\left(
                                          \begin{array}{c}
                                            k \\
                                            j \\
                                          \end{array}
                                        \right)p^j(x)\left(\sqrt{r(x)}^{k-j}+(-\sqrt{r(x)})^{k-j}\right)\,.$$
As $\sqrt{r(x)}^{k-j}+(-\sqrt{r(x)})^{k-j}$ is zero whenever $k-j$
is odd and it is a polynomial whenever $k-j$ is even, there follows
that $\mathbb{E}_1f+\mathbb{E}_2f$ is a polynomial, thus proving
assertion (b).

To prove assertion (c) let us note that
$(\mathbb{E}_1f)^2+(\mathbb{E}_2f)^2=(\mathbb{E}_1f+\mathbb{E}_2f)^2-2\mathbb{E}_1f\mathbb{E}_2f$,
that is,
$$(\mathbb{E}_1f)^2+(\mathbb{E}_2f)^2=4(\mathbb{M}f)^2-2\mathbb{E}_1f\mathbb{E}_2f\,.$$
Hence, assertion (c) is a consequence of (a) and (b).

For arbitrary functions $f, g$, the following identities take place:
\begin{eqnarray} &&\mathbb{E}_1f\,\mathbb{E}_2g+
\mathbb{E}_1g\,\mathbb{E}_2f=(\mathbb{E}_1g+\mathbb{E}_2g)(\mathbb{E}_1f+\mathbb{E}_2f)-(\mathbb{E}_1(gf)+\mathbb{E}_2(gf))\,, \label{eq:ident-aux-E1E2}\\
&&-\mathbb{D}f\,\mathbb{E}_1g+\mathbb{D}g\,\mathbb{E}_1f=(\mathbb{E}_1f+\mathbb{E}_2f)\mathbb{D}g-\mathbb{D}(gf)\,,
\label{eq:ident-aux-DE1E2}
\end{eqnarray}
as $\mathbb{E}_j(gf)=\mathbb{E}_jg\,\mathbb{E}_jf,\; j=1,2.$

Note that (\ref{eq:ident-aux-E1E2}) also reads as
$$\mathbb{E}_1f\,\mathbb{E}_2g+
\mathbb{E}_1g\,\mathbb{E}_2f=4
\mathbb{M}g\,\mathbb{M}f-2\mathbb{M}(gf)\,.$$ Taking into account
the property (\ref{eq:M-prod}), the above equation yields
(\ref{eq:identE1E2}).

Eq. (\ref{eq:ident-aux-DE1E2}) gives us (\ref{eq:ident-D-DE1E2}).
\end{proof}

\subsection{Laguerre-Hahn orthogonal polynomials and auxiliary results}
We shall consider formal orthogonal polynomials related to a
(formal) Stieltjes function defined by
\begin{equation}
S(x)=\displaystyle\sum_{n=0}^{+\infty}u_n
x^{-n-1}\label{eq:defi-S}\end{equation} where $(u_n)$, the sequence
of moments, is such that $\det\left[
\begin{matrix}u_{i+j}\end{matrix}\right]_{i,j=0}^n \neq 0,\, n \geq 0$, $u_0=1.$ The orthogonal
polynomials related to $S$, $P_n, n \geq 0,$ are the diagonal Pad\'e
denominators of (\ref{eq:defi-S}), thus the numerator polynomial (of
degree $n-1$), henceforth denoted by $P_{n-1}^{(1)}$ , and the
denominator $P_n$ (of degree $n$) are determined through
\begin{equation}S(x)-P_{n-1}^{(1)}(x)/P_n(x)=\mathcal{O}(x^{-2n-1})\,, \quad
x \to \infty\,.\label{eq:Herm-Pade-cond}
\end{equation}
Throughout the paper we consider each $P_n$ monic,
 and we will denote the
sequence of monic polynomials  $\{P_n\}_{n\geq 0}$ by SMOP.

Monic orthogonal polynomials satisfy  a three term recurrence
relation \cite{szego}
\begin{equation}
P_{n+1}(x)=(x-\beta_n)P_n(x)-\gamma_nP_{n-1}(x)\,,\quad
n=0,1,2,...\,, \label{eq:ttrr-Pn}
\end{equation} with $P_{-1}(x)=0, \;
P_0(x)=1,$ and $\gamma_n\neq 0, \; n \geq 1, \; \gamma_0=u_0=1$.

The sequence  $\{P_n^{(1)}\}_{n\geq 0}$, also known as the sequence
of associated polynomials of the first kind, satisfies the three
term recurrence relation
\begin{equation}
P^{(1)}_{n}(x)=(x-\beta_n)P^{(1)}_{n-1}(x)-\gamma_nP^{(1)}_{n-2}(x)\,,
\quad n=1,2,... \label{eq:ttrr-Pn1}
\end{equation} with  $P^{(1)}_{-1}(x)=0, \;P^{(1)}_0(x)=1$.

An  equivalent form of  (\ref{eq:Herm-Pade-cond}), often encountered
in the literature of orthogonal polynomials (see, for example,
\cite{walter} and its list of references), is given by
\begin{equation}
q_{n}=P_{n}S-P_{n-1}^{(1)}\,, \quad n \geq 1\,, \;\; q_0=S\,,
\label{eq:defi-qn}
\end{equation}
where $q_n, n \geq 0,$ are the so-called functions of the second
kind corresponding to $\{P_n\}_{n\geq 0}$. The sequence
$\{q_n\}_{n\geq 0}$ also satisfies a three term recurrence relation,
\begin{equation}
q_{n+1}(x)=(x-\beta_n)q_n(x)-\gamma_nq_{n-1}(x)\,,\quad
n=0,1,2,...\label{eq:ttrr-qn}
\end{equation} with initial conditions $q_{-1}=1,\; q_0(x)=S(x).$
There holds $q_n(x)=\mathcal{O}(x^{-n-1}), x \to \infty$.

We will make use of the following result (see \cite{brez}).

\begin{Lemma}
\label{Lemma:liouv} Let $\{P_n\}_{n\geq 0}$ be a SMOP and let
$\{P_n^{(1)}\}_{n\geq 0}$ be the sequence of associated polynomials
of the first kind. The following holds:
\begin{equation}
\mathbb{E}_jP_n^{(1)}\,\mathbb{E}_jP_n-\mathbb{E}_jP_{n+1}\,\mathbb{E}_jP_{n-1}^{(1)}=\prod_{k=0}^n\gamma_k\,,\;\;\;
j=1,2\,, \quad n\geq 0\,.\label{eq:liouville}
\end{equation}
Therefore, for each $j=1,2,$ $\mathbb{E}_jP_n^{(1)}$ and
$\mathbb{E}_jP_{n+1}$ do not share zeroes.
\end{Lemma}
\begin{proof}
Eq. (\ref{eq:liouville}) follows from the application of the
operator $\mathbb{E}_j$, $j=1,2,$ to the identity
\begin{equation}P_n^{(1)}P_n-P_{n+1}P_{n-1}^{(1)}=\prod_{k=0}^n\gamma_k\,,
\quad n\geq 0\,.\label{eq:liou-Pn}\nonumber\end{equation} From
(\ref{eq:liouville}) there follows the statement concerning the
zeros.
\end{proof}

\begin{Definition}
A SMOP $\{P_n\}_{n\geq 0}$ related to a Stieltjes function, $S$, is
said to be Laguerre-Hahn if $S$ satisfies a Riccati equation
\begin{equation}
A(x)(\mathbb{D}S)(x)=B(x)(\mathbb{E}_1S)(x)(\mathbb{E}_2S)(x)+C(x)(\mathbb{M}S)(x)+D(x)\,,\label{eq:defi-Ric-S}
\end{equation}
where  $A, B, C, D$ are polynomials in $x$,  $A\neq 0$. \\
If $B\equiv 0$, then $\{P_n\}_{n\geq 0}$ is said to be
semi-classical.
\end{Definition}

We will make use of the Lemma that follows.

\begin{Lemma}\label{Theorem:Magnus}
Let  $\{f_n\}$ be  a sequence of functions satisfying a recurrence
relation
\begin{equation} f_{n+1}(x)=(x-\beta_n)f_n(x)-\gamma_nf_{n-1}(x)\,,
\;\; \gamma_n\neq 0, \; n \geq 0\,.\label{eq:ttrr-f}
\end{equation}
 Let $g_n=f_{n+1}/f_n$ satisfy
\begin{equation}
A_n(x)(\mathbb{D}g_n)(x)=B_n(x)(\mathbb{E}_1 g_n)(x) (\mathbb{E}_2
g_n)(x)+C_n(x) (\mathbb{M}g_n)(x)+D_n(x), \quad n\geq
0\,,\label{eq:ric-gn}
\end{equation}
where $A_n$, $B_n$, $C_n$, $D_n$ are polynomials whose degrees are
uniformly bounded. For each $n\geq 0,$ there exists a polynomial,
$\varrho_n$, with uniformly bounded degree, such that the following
relations hold:
\begin{eqnarray}
A_{n+1}&=&\varrho_n\left(A_n-\frac{\Delta_y^2}{2}\,\frac{D_n}{\gamma_{n+1}}\right)\,,\label{eq:Magnus-An}\\
B_{n+1}&=&\varrho_n\frac{D_n}{\gamma_{n+1}}\,, \label{eq:Magnus-Bn}\\
C_{n+1}&=&\varrho_n\left(-C_n-2\mathbb{M}(x-\beta_{n+1})\frac{D_n}{\gamma_{n+1}}\right)\,, \label{eq:Magnus-Cn}\\
D_{n+1}&=&\varrho_n\left(A_n+\gamma_{n+1}B_n+\mathbb{M}(x-\beta_{n+1})C_n
+\mathbb{E}_1(x-\beta_{n+1})\mathbb{E}_2(x-\beta_{n+1})\frac{D_n}{\gamma_{n+1}}\right).\hspace{0.4cm}
\label{eq:Magnus-Dn}
\end{eqnarray}
\end{Lemma}
\begin{proof}
From (\ref{eq:ttrr-f}) we have
$g_n=(x-\beta_n)-\gamma_n\displaystyle\frac{1}{g_{n-1}},$ thus,
\begin{equation}
g_{n+1}=(x-\beta_{n+1})-\gamma_{n+1}\frac{1}{g_{n}}\,.\label{eq:Magnusaux-1}
\end{equation}

Applying $\mathbb{D}$ to (\ref{eq:Magnusaux-1}) and using
$\mathbb{D}(1/g_{n})=-\displaystyle\frac{\mathbb{D}g_{n}}{\mathbb{E}_1g_{n}\,\mathbb{E}_2g_{n}}$
(cf. (\ref{eq:D-1sobref}))
 we get
$$\mathbb{D}g_{n+1}=1+\gamma_{n+1}\frac{\mathbb{D}g_{n}}{\mathbb{E}_1g_{n}\,\mathbb{E}_2g_{n}}\,.$$
Now we  multiply the above equation by $A_n$ and use
(\ref{eq:ric-gn}), as well as the identity
$\displaystyle\frac{\mathbb{M}g_n}{\mathbb{E}_1g_n\,
\mathbb{E}_2g_n}=\mathbb{M}(1/g_n) $ (cf. (\ref{eq:M-1sobref})),
thus obtaining
\begin{equation}
A_n\mathbb{D}g_{n+1}=A_n+\gamma_{n+1}B_n+\gamma_{n+1}C_n\mathbb{M}(1/g_n)
+\frac{\gamma_{n+1}D_n}{\mathbb{E}_1g_{n}\,\mathbb{E}_2g_{n}}\,.\label{eq:Magnusaux-2}
\end{equation}

Note that from (\ref{eq:Magnusaux-1}) we get
\begin{equation}
\mathbb{M}(1/g_n)=\frac{\mathbb{M}(x-\beta_{n+1})}{\gamma_{n+1}}-\frac{\mathbb{M}g_{n+1}}{\gamma_{n+1}}\,.\label{eq:Magnusaux-3}
\end{equation}
Also, as
$\displaystyle\frac{1}{\mathbb{E}_1\,g_{n}\mathbb{E}_2g_{n}}=\displaystyle\mathbb{E}_1\left(\frac{1}{g_n}\right)\mathbb{E}_2\left(\frac{1}{g_n}\right),$
from (\ref{eq:Magnusaux-1}) we get
\begin{multline*}\frac{1}{\mathbb{E}_1g_{n}\,\mathbb{E}_2g_{n}}=
\frac{1}{\gamma^2_{n+1}}\left(\mathbb{E}_1(x-\beta_{n+1})\mathbb{E}_2(x-\beta_{n+1})\right.\\
-\left.
\left(\mathbb{E}_1(x-\beta_{n+1})\mathbb{E}_2g_{n+1}+\mathbb{E}_2(x-\beta_{n+1})\mathbb{E}_1g_{n+1}\right)
+\mathbb{E}_1g_{n+1} \mathbb{E}_2g_{n+1}\right)\,,\end{multline*}
thus, taking into account (\ref{eq:identE1E2}), we get
\begin{multline}
\displaystyle\frac{1}{\mathbb{E}_1\,g_{n}\mathbb{E}_2g_{n}}=\frac{1}{\gamma^2_{n+1}}\mathbb{E}_1(x-\beta_{n+1})\mathbb{E}_2(x-\beta_{n+1})\\
-\frac{1}{\gamma^2_{n+1}}
\left(2\mathbb{M}(x-\beta_{n+1})\mathbb{M}g_{n+1}-\frac{\Delta_y^2}{2}\,\mathbb{D}g_{n+1}\right)
+\frac{\mathbb{E}_1g_{n+1}\,\mathbb{E}_2g_{n+1}}{\gamma^2_{n+1}}\,,
\label{eq:Magnusaux-4}
\end{multline}
where is was used $\mathbb{D}(x-\beta_{n+1})=1$. The substitution of
(\ref{eq:Magnusaux-3}) and (\ref{eq:Magnusaux-4}) in
(\ref{eq:Magnusaux-2}) yields
\begin{equation}\hat{A}_{n+1}\mathbb{D}g_{n+1}=\hat{B}_{n+1}\mathbb{E}_1g_{n+1}\,\mathbb{E}_2
g_{n+1}+\hat{C}_{n+1}\mathbb{M}g_{n+1}+\hat{D}_{n+1}\,,\label{eq:ric-hat}\end{equation}
with
\begin{eqnarray}
\hat{A}_{n+1}&=&A_n-\frac{\Delta_y^2}{2}\,\frac{D_n}{\gamma_{n+1}}\,,\label{eq:Magnus-hat-An}\nonumber\\
\hat{B}_{n+1}&=&\frac{D_n}{\gamma_{n+1}}\,, \label{eq:Magnus-hat-Bn}\nonumber\\
\hat{C}_{n+1}&=&-C_n-2\mathbb{M}(x-\beta_{n+1})\frac{D_n}{\gamma_{n+1}}\,, \label{eq:Magnus-hat-Cn}\nonumber\\
\hat{D}_{n+1}&=&A_n+\gamma_{n+1}B_n+\mathbb{M}(x-\beta_{n+1})C_n
+\mathbb{E}_1(x-\beta_{n+1})\mathbb{E}_2(x-\beta_{n+1})\frac{D_n}{\gamma_{n+1}}\,.
\label{eq:Magnus-hat-Dn}\nonumber
\end{eqnarray}

 Taking into account (\ref{eq:ric-hat}) and (\ref{eq:ric-gn}) written to
 $n+1$, there follows the existence of a polynomial, say $\varrho_n,$ such
 that $$\frac{{A}_{n+1}}{\hat{A}_{n+1}}=\frac{{B}_{n+1}}{\hat{B}_{n+1}}=\frac{{C}_{n+1}}{\hat{C}_{n+1}}=\frac{{D}_{n+1}}{\hat{D}_{n+1}}=\varrho_{n}\,,\quad n \geq 0\,,$$
 thus
 (\ref{eq:Magnus-An})-(\ref{eq:Magnus-Dn}) follow.  As the degrees of  $A_n, B_n, C_n,
 D_n,$ are bounded by a number independent of $n$,
 there follows that $\deg(\varrho_n)$  must also be bounded.
\end{proof}

\section{Characterization theorem}\label{sec:3}
\begin{Theorem} \label{teo2}
Let $S$ be a Stieltjes function, let $\{P_n\}_{n\geq 0}$ be the
corresponding SMOP, and let $\{P_{n}^{(1)}\}_{n\geq 0}$,
$\{q_n\}_{n\geq 0}$ be the sequence of associated polynomials of the
first kind and the sequence of functions of the second kind,
respectively.
 The following statements are
equivalent:\\
(a)\; $S$ satisfies the Riccati equation (\ref{eq:defi-Ric-S}),
$$A \mathbb{D}S =B \mathbb{E}_1S\,\mathbb{E}_2S+C\mathbb{M}S+D\,,$$
where $A, B, C, D$ are polynomials;\\
 (b)\;
$P_n$ and  $P_{n}^{(1)}$ satisfy the difference relations, for all
$n \geq 1,$
\begin{equation}
\begin{cases}
A\mathbb{D}P_{n}=(l_{n-1}+\Delta_y
\pi_{n-1})\mathbb{E}_1P_{n}-\displaystyle
C/2\,\mathbb{E}_2P_{n}-B\mathbb{E}_2P_{n-1}^{(1)}+\Theta_{n-1}\mathbb{E}_1P_{n-1}\,,
\vspace{0.2cm}\\
A\mathbb{D}P^{(1)}_{n-1}=(l_{n-1}+\Delta_y
\pi_{n-1})\mathbb{E}_1P_{n-1}^{(1)}+\displaystyle
C/2\,\mathbb{E}_2P^{(1)}_{n-1}+D\mathbb{E}_2P_{n}+
\Theta_{n-1}\mathbb{E}_1P_{n-2}^{(1)}\,;
\end{cases}\label{eq:est-psin1-(1)}
\end{equation}
(c)\; $q_n$ satisfies
\begin{equation} A \mathbb{D}q_{n}
   =(l_{n-1}+\Delta_y \pi_{n-1})\mathbb{E}_1 q_{n}+\left(B\mathbb{E}_1S+C/2\right)\mathbb{E}_2q_{n}+
  \Theta_{n-1}\,\mathbb{E}_1q_{n-1}\,, \;\; n\geq
  0\,,\label{eq:est-qn-(1)}
\end{equation}
where $\Delta_y=y_2-y_1$, and
 $l_n, \pi_n, \Theta_n$ are polynomials of uniformly bounded degrees.
\end{Theorem}
The proof of Theorem \ref{teo2} will be given in section 4.

\vspace{0.1cm}

\noindent{\it{Remark }}.\; The characterizations stated in Theorem
\ref{teo2} are not uniquely represented. One can also deduce that
the
following statements (a), (b), (c) are equivalent:\\
(a)\; $S$ satisfies the Riccati equation (\ref{eq:defi-Ric-S}), $$A
\mathbb{D}S =B \mathbb{E}_1S\,\mathbb{E}_2S+C\mathbb{M}S+D\,;$$
 (b)\; $P_n$ and  $P_{n}^{(1)}$ satisfy the difference relations,
 for all $n \geq 1,$
\begin{equation}
\begin{cases}
A\mathbb{D}P_{n}=(l_{n-1}-\Delta_y
\pi_{n-1})\mathbb{E}_2P_{n}-\displaystyle
C/2\,\mathbb{E}_1P_{n}-B\mathbb{E}_1P^{(1)}_{n-1}+\Theta_{n-1}\mathbb{E}_2P_{n-1}\,,
\vspace{0.2cm}\\
A\mathbb{D}P^{(1)}_{n-1}=(l_{n-1}-\Delta_y
\pi_{n-1})\mathbb{E}_2P_{n-1}^{(1)}+\displaystyle
C/2\,\mathbb{E}_1P^{(1)}_{n-1}+D\mathbb{E}_1P_{n}
+\Theta_{n-1}\mathbb{E}_2P_{n-2}^{(1)}\,;
\end{cases}\label{eq:est-psin1-(2)}
\end{equation}
(c)\; $q_n$ satisfies
\begin{equation} A \mathbb{D}q_{n}
   =(l_{n-1}-\Delta_y \pi_{n-1})\,\mathbb{E}_2 q_{n}+\left(B\mathbb{E}_2S+ C/2\right)\mathbb{E}_1q_{n}+ \Theta_{n-1}\,\mathbb{E}_2q_{n-1}\,,
   \;\;
   n\geq0\,.
\label{eq:est-qn-(2)}
\end{equation}

Therefore, we deduce the result that follows.
\begin{Theorem}\label{Theorem:teo-gather}
Let $S$ be a Stieltjes function satisfying the Riccati equation
(\ref{eq:defi-Ric-S}),
$$A \mathbb{D}S =B \mathbb{E}_1S\,\mathbb{E}_2S+C\mathbb{M}S+D\,,$$
where $A, B, C, D$ are polynomials.   Let $\{P_n\}_{n\geq 0}$ be the
SMOP related to $S$, and let  $\{P_{n}^{(1)}\}_{n\geq 0}$,
$\{q_n\}_{n\geq 0}$ be the sequence of associated polynomials of the
first kind and the sequence of functions of the second kind,
respectively. The following relations hold, for all $n\geq 0$:
\begin{eqnarray}
&&A_{n+1}\mathbb{D}P_{n+1}=(l_n-C/2)\mathbb{M}P_{n+1}-B\mathbb{M}P_{n}^{(1)}+\Theta_n\mathbb{M}P_n\,,\label{eq:est-conseq-Pn}\\
&&A_{n+1}\mathbb{D}P^{(1)}_{n}=(l_n+C/2)\mathbb{M}P^{(1)}_{n}+D\mathbb{M}P_{n+1}+\Theta_n\mathbb{M}P^{(1)}_{n-1}\,,\label{eq:est-conseq-Pn1}\\
&&A_n\mathbb{D}q_{n}=(l_{n-1}+C/2)\mathbb{M}q_{n}+B\left(2\mathbb{M}S\,
\mathbb{M}q_n-\mathbb{M}(Sq_n)\right)+\Theta_{n-1}\mathbb{M}q_{n-1}
\,,\label{eq:est-qn-conseq}
\end{eqnarray}
with $A_n, l_n, \Theta_n, \pi_{n}$ uniformly bounded degree
polynomials, $A_n$ given by $A_n=A+\frac{\Delta^2_y}{2}\,\pi_{n-1}$.
\end{Theorem}
\begin{proof}
The sum of the first equation  in  (\ref{eq:est-psin1-(1)}) with the
first equation in (\ref{eq:est-psin1-(2)}) and the division of the
resulting equation by $2$ gives us
$$A\mathbb{D}P_{n}=(l_{n-1}-C/2)\mathbb{M}P_{n}-\frac{\pi_{n-1}}{2}\Delta_y \, (\mathbb{E}_2P_n-\mathbb{E}_1P_n)-B\mathbb{M}P_{n-1}^{(1)}
+\Theta_{n-1}\mathbb{M}P_{n-1}\,, \quad n \geq 1\,. $$ Note that as
$\Delta_y=y_2-y_1$, then
$$ \Delta_y\, (\mathbb{E}_2P_n-\mathbb{E}_1P_n)=\Delta_y^2\,\mathbb{D}P_n\,.$$
Thus, we obtain
$$(A+\frac{\Delta^2_y}{2}\,\pi_{n-1})\mathbb{D}P_{n}=(l_{n-1}-C/2)\mathbb{M}P_{n}-B\mathbb{M}P_{n-1}^{(1)}
+\Theta_{n-1}\mathbb{M}P_{n-1}\,, \quad n \geq 1\,,$$ hence
(\ref{eq:est-conseq-Pn}).

The equations (\ref{eq:est-conseq-Pn1}) and (\ref{eq:est-qn-conseq})
follow by an analogue manner.
\end{proof}

\noindent{\it{Remark }}.\; The equations
(\ref{eq:est-conseq-Pn})-(\ref{eq:est-qn-conseq}) extend the ones
given in \cite{witte} for the semi-classical case.

\begin{Corollary}
 The polynomials $l_n, \Theta_n$ of Theorems \ref{teo2} and \ref{Theorem:teo-gather} satisfy,
 for all $n\geq 0,$
\begin{eqnarray}
&&\pi_{n+1}+\pi_{n}=-\frac{\Theta_n}{2\gamma_{n+1}}-\sum_{k=0}^{n}\frac{\Theta_{k-1}}{\gamma_k}\,,\label{eq:prop-ln-Thetan-1}\\
&&l_{n+1}+l_{n}+\mathbb{M}(x-\beta_{n+1})\frac{\Theta_{n}}{\gamma_{n+1}}=0\,,
\label{eq:prop-ln-Thetan-2}\\
&&A+\frac{\Delta_y^2}{2}(\pi_{n}+\pi_{n-1})+\frac{\Theta_{n-1}}{\gamma_n}\left(-\frac{\Delta_y^2}{4}+\gamma_{n+1}-\mathbb{M}(x-\beta_n)\mathbb{M}(x-\beta_{n+1})\right)
\nonumber\\
&&\hspace{0.5cm}+\frac{\Theta_{n}}{\gamma_{n+1}}\mathbb{E}_1(x-\beta_{n+1})\mathbb{E}_2(x-\beta_{n+1})+\mathbb{M}(x-\beta_{n+1})(l_n-l_{n-1})
=\Theta_{n+1}\,,\hspace{0.5cm} \label{eq:prop-ln-Thetan-3}
\end{eqnarray}
with initial conditions $\pi_{-1}=0,$ $\pi_{0}=-D/2,\; l_{-1}=C/2,\;
l_{0}=-\mathbb{M}(x-\beta_{0})D-C/2,\;$ $\Theta_{-1}=D,\;$ $
\Theta_0=A-\frac{\Delta_y^2}{4}D-(l_0-C/2)\mathbb{M}(x-\beta_{0})+B$.
\end{Corollary}
\begin{proof}
 Multiply
(\ref{eq:est-qn-(1)}), written to $n+1$, by $\mathbb{E}_2q_n$, and
subtract to (\ref{eq:est-qn-(1)}) multiplied by
$\mathbb{E}_2q_{n+1}$. Then, multiply the resulting equation by
$1/(\mathbb{E}_1q_n\mathbb{E}_2q_n)$, thus obtaining
\begin{multline}
A\,\mathbb{D}\left(\frac{q_{n+1}}{q_n}\right)=l_n\mathbb{E}_1\left(\frac{q_{n+1}}{q_n}\right)-l_{n-1}\mathbb{E}_2\left(\frac{q_{n+1}}{q_n}\right)
+\Theta_n\\-\Theta_{n-1}\mathbb{E}_1\left(\frac{q_{n-1}}{q_n}\right)\mathbb{E}_2\left(\frac{q_{n+1}}{q_n}\right)+
\Delta_y\left\{\pi_n\mathbb{E}_1\left(\frac{q_{n+1}}{q_n}\right)-\pi_{n-1}\mathbb{E}_2\left(\frac{q_{n+1}}{q_n}\right)
\right\}\,, \label{cor:aux1}
\end{multline}
where we used  the property (\ref{eq:D-quoc-E2}).

 From the recurrence relation for
$q_n$ one has
$\displaystyle\frac{q_{n-1}}{q_n}=\frac{(x-\beta_n)}{\gamma_n}-\frac{1}{\gamma_n}\frac{q_{n+1}}{q_n}\,,$
thus
\begin{equation}
\mathbb{E}_1\left(\frac{q_{n-1}}{q_n}\right)=\frac{1}{\gamma_n}\mathbb{E}_1(x-\beta_n)-\frac{1}{\gamma_n}\mathbb{E}_1\left(\frac{q_{n+1}}{q_n}\right)
\,. \label{eq:cor-aux2}
\end{equation}
Let us substitute (\ref{eq:cor-aux2}) in (\ref{cor:aux1}). Using the
notation $g_{n}=\displaystyle q_{n+1}/{q_n}$, the resulting equation
is given by
\begin{multline}
A\,\mathbb{D}g_n=l_n\mathbb{E}_1g_n-l_{n-1}\mathbb{E}_2g_n
+\Theta_n-\frac{\Theta_{n-1}}{\gamma_n}\left(\mathbb{E}_1(x-\beta_n)-\mathbb{E}_1g_n\right)
\mathbb{E}_2g_n\\+\Delta_y\left(\pi_n\mathbb{E}_1g_n-\pi_{n-1}\mathbb{E}_2g_n\right)\,.
\label{eq:aux4-1}
\end{multline}

On the other hand, if we proceed  analogously as above, but now
using  (\ref{eq:est-qn-(2)}) and the property (\ref{eq:D-quoc-E1}),
we obtain
\begin{multline}
A\,\mathbb{D}g_n=l_n\mathbb{E}_2g_n-l_{n-1}\mathbb{E}_1g_n
+\Theta_n-\frac{\Theta_{n-1}}{\gamma_n}\left(\mathbb{E}_2(x-\beta_n)-\mathbb{E}_2g_n\right)
\mathbb{E}_1g_n\\-\Delta_y\left(\pi_n\mathbb{E}_2g_n-\pi_{n-1}\mathbb{E}_1g_n\right)\,.
\label{eq:aux4-2}
\end{multline}
By summing  (\ref{eq:aux4-1}) with (\ref{eq:aux4-2}) and by dividing
the resulting equation by $2$ we obtain
\begin{multline}
A\,\mathbb{D}g_n
=(l_n-l_{n-1})\mathbb{M}g_n+\Theta_n-\frac{\Theta_{n-1}}{2\gamma_n}\left(\mathbb{E}_1(x-\beta_n)\,\mathbb{E}_2g_n+\mathbb{E}_2(x-\beta_n)\,
\mathbb{E}_1g_n\right)\\+\frac{\Theta_{n-1}}{\gamma_n}\mathbb{E}_1g_n\mathbb{E}_2g_n
-(\pi_{n-1}+\pi_{n})\,\frac{\Delta_y}{2}\,\left(\mathbb{E}_2g_n-\mathbb{E}_1g_n\right)\,.
\label{cor:aux4-3}
\end{multline}

Taking into account  (\ref{eq:identE1E2}) one has
\begin{equation}\mathbb{E}_1(x-\beta_n)\,\mathbb{E}_2g_n+\mathbb{E}_2(x-\beta_n)\,
\mathbb{E}_1g_n=2\mathbb{M}(x-\beta_n)\,\mathbb{M}g_n-\frac{\Delta_y^2}{2}\mathbb{D}g_n\,,\label{eq:aux-E1E2-g}\end{equation}
where it was used $\mathbb{D}(x-\beta_n)=1.$ The use of
(\ref{eq:aux-E1E2-g}) as well as the use of
$$\Delta_y\,\left(\mathbb{E}_2g_n-\mathbb{E}_1g_n\right)=\Delta_y^2\mathbb{D}g_n$$
in (\ref{cor:aux4-3}) yields the Riccati equation
\begin{equation}
A_n\,\mathbb{D}g_{n}=B_n\mathbb{E}_1g_n\,\mathbb{E}_2g_n+C_n\mathbb{M}g_n+D_n
\label{cor:aux-5}\nonumber
\end{equation}
with
\begin{eqnarray}
&&A_n=A+\frac{\Delta_y^2}{2}\left(\pi_{n}+\pi_{n-1}-\frac{\Theta_{n-1}}{2\gamma_n}\right)\,,\label{eq:ric-gn-qA}\\
&&B_n=\frac{\Theta_{n-1}}{\gamma_n}\,,\\
&&C_n=l_n-l_{n-1}-\mathbb{M}(x-\beta_n)\frac{\Theta_{n-1}}{\gamma_n}\,,\\
&&D_n=\Theta_n\,.\label{eq:ric-gn-qD}
\end{eqnarray}

 Taking into account the Lemma \ref{Theorem:Magnus}, for each $n \geq 0$,
there exists a polynomial, $\varrho_n$, such that  the above
polynomials $A_n, B_n, C_n, D_n$ satisfy
(\ref{eq:Magnus-An})-(\ref{eq:Magnus-Dn}). Recall the equation
(\ref{eq:Magnus-Bn}), $B_{n+1}=\varrho_n D_n/\gamma_{n+1}$, which
yields $\Theta_n/\gamma_{n+1}=\varrho_n \Theta_n/\gamma_{n+1}$, thus
$\varrho_n=1.$ Therefore, from (\ref{eq:Magnus-An}),
(\ref{eq:Magnus-Cn})  and (\ref{eq:Magnus-Dn}) with $\varrho_n=1$ we
obtain, respectively, for all $n \geq 0$:
\begin{eqnarray}
&&\pi_{n+1}+\frac{\Theta_{n}}{2\gamma_{n+1}}=\pi_{n-1}-\frac{\Theta_{n-1}}{2\gamma_{n}}\,,\label{eq:result-An-1}\\
&&l_{n+1}+\mathbb{M}(x-\beta_{n+1})\frac{\Theta_{n}}{\gamma_{n+1}}=l_{n-1}+\mathbb{M}(x-\beta_n)\frac{\Theta_{n-1}}{\gamma_n}\,,\label{eq:result-Cn-1}\\
&&A+\frac{\Delta_y^2}{2}(\pi_{n}+\pi_{n-1})+\frac{\Theta_{n-1}}{\gamma_n}\left(-\frac{\Delta_y^2}{4}+\gamma_{n+1}-\mathbb{M}(x-\beta_n)\mathbb{M}(x-\beta_{n+1})\right)
\nonumber\\
&&\hspace{0.5cm}+\frac{\Theta_{n}}{\gamma_{n+1}}\mathbb{E}_1(x-\beta_{n+1})\mathbb{E}_2(x-\beta_{n+1})+\mathbb{M}(x-\beta_{n+1})(l_n-l_{n-1})
=\Theta_{n+1}\,.\hspace{0.5cm}
 \label{eq:result-Dn-1}
\end{eqnarray}

Note that (\ref{eq:result-Dn-1}) is (\ref{eq:prop-ln-Thetan-3}). It
only remains to analyze (\ref{eq:result-An-1}) and
(\ref{eq:result-Cn-1}). Let us add $\pi_n$ in both hand sides of
(\ref{eq:result-An-1}), and add $l_n$ in both hand sides of
(\ref{eq:result-Cn-1}). We obtain, respectively,
\begin{equation}T_{n+1}=T_n-\frac{\Theta_{n-1}}{\gamma_n}\,,\;\;\;\;
L_{n+1}=L_n\,,\quad n\geq 0\,,\label{eq:telesc}\end{equation} with
$$T_{n}=\pi_{n}+\pi_{n-1}+\frac{\Theta_{n-1}}{2\gamma_{n}}\,,
\;\;\;L_{n}=l_{n}+l_{n-1}
+\mathbb{M}(x-\beta_{n})\frac{\Theta_{n-1}}{\gamma_{n}}\,,\quad n
\geq 0\,.$$ From (\ref{eq:telesc}) we obtain
$$T_{n+1}=T_0-\sum_{k=0}^{n}\frac{\Theta_{k-1}}{\gamma_k}\,,
\;\;\;\; L_{n+1}=L_0\,.$$ Note that $T_0=\pi_{0}+D/2,\;
L_0=l_0+C/2+\mathbb{M}(x-\beta_{0})D,$ as $ \pi_{-1}=0,
\Theta_{-1}=D, \gamma_0=1, l_{-1}=C/2.$

Putting $n=1$ in the second equation of (\ref{eq:est-psin1-(1)}) we
obtain $\Delta_y \pi_0+(y_2-\beta_0)D=-l_0-C/2,$ which, using
(\ref{eq:y1-y2}), gives $\sqrt{r}(2\pi_0+D)=-(p-\beta_0)D-l_0-C/2,$
thus, $2\pi_0+D=-(p-\beta_0)D-l_0-C/2=0,$ hence $\pi_0=-D/2$. Thus,
 $T_0=0$, and (\ref{eq:prop-ln-Thetan-1}) follows.

Putting $n=0$ in (\ref{eq:est-conseq-Pn1}) we get  $L_0=0,$ and
(\ref{eq:prop-ln-Thetan-2})
 follows.

To find $\Theta_0$ put $n=0$ in (\ref{eq:est-conseq-Pn}).
\end{proof}

\section{Proof of Theorem \ref{teo2}}\label{sec:4}

\vspace{0.2cm}

\noindent{\bf{Proof\; of\;$(a) \Rightarrow (b)$.}}\\
If we use  $S=\displaystyle
\frac{q_{n}}{P_{n}}+\frac{P_{n-1}^{(1)}}{P_{n}},\;\; n\geq 1$ (cf.
(\ref{eq:defi-qn})), then (\ref{eq:defi-Ric-S}) yields
\begin{equation}
M_n=-A\mathbb{D}\left(\displaystyle
\frac{P^{(1)}_{n-1}}{P_{n}}\right)+B\mathbb{E}_1\left(\displaystyle
\frac{P^{(1)}_{n-1}}{P_{n}}\right)\mathbb{E}_2\left(\displaystyle
\frac{P^{(1)}_{n-1}}{P_{n}}\right)+C\mathbb{M}\left(\displaystyle
\frac{P^{(1)}_{n-1}}{P_{n}}\right)+D\,,\label{eq:teo1-aux1}
\end{equation}
where \begin{multline*}
M_n=A\mathbb{D}\left(\displaystyle\frac{q_{n}}{P_{n}}\right)-B\left[\mathbb{E}_1\left(\displaystyle
\frac{q_{n}}{P_{n}}\right)\mathbb{E}_2\left(\displaystyle
\frac{q_{n}}{P_{n}}\right)+\mathbb{E}_1\left(\displaystyle
\frac{q_{n}}{P_{n}}\right)\mathbb{E}_2\left(\displaystyle
\frac{P^{(1)}_{n-1}}{P_{n}}\right)+\right.\\
\phantom{ola}\hspace{3.5cm}\left.\mathbb{E}_1\left(\displaystyle
\frac{P^{(1)}_{n-1}}{P_{n}}\right)\mathbb{E}_2\left(\displaystyle
\frac{q_{n}}{P_{n}}\right)\right] -C\mathbb{M}\left(\displaystyle
\frac{q_{n}}{P_{n}}\right)\,.\end{multline*}

 By
multiplying both hand sides of (\ref{eq:teo1-aux1}) by
$\mathbb{E}_1P_{n}\,\mathbb{E}_2P_{n}$ and using the properties
(\ref{eq:M-quoc}) and (\ref{eq:D-quoc-E1}), we obtain
\begin{multline*}
M_n\mathbb{E}_1P_{n}\,\mathbb{E}_2P_{n}=
-A\mathbb{D}P_{n-1}^{(1)}\,\mathbb{E}_1P_{n}+A\mathbb{D}P_{n}\,\mathbb{E}_1P^{(1)}_{n-1}+B\mathbb{E}_1P^{(1)}_{n-1}\,\mathbb{E}_2P_{n-1}^{(1)}\\
+\frac{C}{2}\left(\mathbb{E}_1P^{(1)}_{n-1}\,\mathbb{E}_2P_{n}+\mathbb{E}_1P_{n}\,\mathbb{E}_2P_{n-1}^{(1)}\right)+D\mathbb{E}_1P_{n}\,\mathbb{E}_2P_{n}\,.
\end{multline*}
Taking into account the Lemma \ref{eq:lemma-ident-aux}, there holds
that
$-\mathbb{D}P_{n-1}^{(1)}\,\mathbb{E}_1P_{n}+\mathbb{D}P_{n}\,\mathbb{E}_1P^{(1)}_{n-1}$,
$\mathbb{E}_1P^{(1)}_{n-1}\,\mathbb{E}_2P_{n-1}^{(1)},
\mathbb{E}_1P^{(1)}_{n-1}\,\mathbb{E}_2P_{n}+\mathbb{E}_1P_{n}\,\mathbb{E}_2P_{n-1}^{(1)}$
and $\mathbb{E}_1P_{n}\,\mathbb{E}_2P_{n} $ are polynomials (in
$x$), thus the above right hand side is a polynomial (in $x$).
Hence, let us write
\begin{multline}
-A\mathbb{D}P_{n-1}^{(1)}\mathbb{E}_1P_{n}+A\mathbb{D}P_{n}\mathbb{E}_1P^{(1)}_{n-1}+B\mathbb{E}_1P^{(1)}_{n-1}\mathbb{E}_2P_{n-1}^{(1)}\\
+\frac{C}{2}\left(\mathbb{E}_1P^{(1)}_{n-1}\mathbb{E}_2P_{n}+\mathbb{E}_1P_{n}\mathbb{E}_2P_{n-1}^{(1)}\right)+D\mathbb{E}_1P_{n}\mathbb{E}_2P_{n}=
\hat{\Theta}_{n-1}\,.\label{eq:theta-hat}
\end{multline}
Note that  $\hat{\Theta}_{n-1}$ has bounded degree,  as
$(q_n/P_n)(x)=\mathcal{O}(x^{-2n-1}), x \to \infty$. One has
$\deg(\hat{\Theta}_{n-1})\leq \max\{\deg(A)-2, \deg(B)-2,
\deg(C)-1\}.$

Taking into account
$\mathbb{E}_1P_{n-1}^{(1)}\,\mathbb{E}_1P_{n-1}-\mathbb{E}_1P_{n}\,\mathbb{E}_1P_{n-2}^{(1)}=\prod_{k=0}^{n-1}\gamma_k,
\;n\geq 1$ (cf. (\ref{eq:liouville})),  (\ref{eq:theta-hat}) can be
written as
\begin{multline}
-A\mathbb{D}P_{n-1}^{(1)}\,\mathbb{E}_1P_{n}+A\mathbb{D}P_{n}\,\mathbb{E}_1P^{(1)}_{n-1}+B\mathbb{E}_1P^{(1)}_{n-1}\,\mathbb{E}_2P_{n-1}^{(1)}\\
+\frac{C}{2}\left(\mathbb{E}_1P^{(1)}_{n-1}\,\mathbb{E}_2P_{n}+\mathbb{E}_1P_{n}\,\mathbb{E}_2P_{n-1}^{(1)}\right)+D\mathbb{E}_1P_{n}\,\mathbb{E}_2P_{n}
\\=
\Theta_{n-1}\left(\mathbb{E}_1P_{n-1}^{(1)}\,\mathbb{E}_1P_{n-1}-\mathbb{E}_1P_{n}\,\mathbb{E}_1P_{n-2}^{(1)}\right)
\,, \label{eq:theta-n}
\end{multline}
where $\Theta_{n-1}=\hat{\Theta}_{n-1}/\prod_{k=0}^{n-1}\gamma_k$.
Upon re-organizing, (\ref{eq:theta-n}) gives us
\begin{multline*}
\left\{A\mathbb{D}P_{n}+C/2\,\mathbb{E}_2P_{n}+B\mathbb{E}_2P_{n-1}^{(1)}-\Theta_{n-1}\mathbb{E}_1P_{n-1}\right\}\mathbb{E}_1P_{n-1}^{(1)}\\
=\left\{A\mathbb{D}P^{(1)}_{n-1}-C/2\,\mathbb{E}_2P^{(1)}_{n-1}-D\mathbb{E}_2P_{n}-\Theta_{n-1}\mathbb{E}_1P_{n-2}^{(1)}\right\}\mathbb{E}_1P_{n}\,.
\end{multline*}
As $\mathbb{E}_1P_{n-1}^{(1)}$ and $\mathbb{E}_1P_{n}$ do not have
common zeroes, for all $n\geq 1,$ then there exists a function,
$L_{n-1},$ such that
\begin{equation*}
\begin{cases}
A\mathbb{D}P_{n}+C/2\,\mathbb{E}_2P_{n}+B\mathbb{E}_2P_{n-1}^{(1)}-\Theta_{n-1}\mathbb{E}_1P_{n-1}=L_{n-1}\mathbb{E}_1P_{n}\,,\\
A\mathbb{D}P^{(1)}_{n-1}-C/2\,\mathbb{E}_2P^{(1)}_{n-1}-D\mathbb{E}_2P_{n}-\Theta_{n-1}\mathbb{E}_1P_{n-2}^{(1)}={L}_{n-1}\mathbb{E}_1P_{n-1}^{(1)}\,,
\end{cases}
\end{equation*}
where, taking into account (\ref{eq:y1-y2}), there holds
${L}_{n-1}=l_{n-1}+\Delta_y \pi_{n-1},$ with $l_{n-1}, \pi_{n-1}$
polynomials (in $x$). Thus, we have
\begin{equation}
\begin{cases}
A\mathbb{D}P_{n}+C/2\,\mathbb{E}_2P_{n}+B\mathbb{E}_2P_{n-1}^{(1)}-\Theta_{n-1}\mathbb{E}_1P_{n-1}=(l_{n-1}+\Delta_y \pi_{n-1})\mathbb{E}_1P_{n}\,,\\
A\mathbb{D}P^{(1)}_{n-1}-C/2\,\mathbb{E}_2P^{(1)}_{n-1}-D\mathbb{E}_2P_{n}-\Theta_{n-1}\mathbb{E}_1P_{n-2}^{(1)}=(l_{n-1}+\Delta_y
\pi_{n-1})\mathbb{E}_1P_{n-1}^{(1)}\,,
\end{cases}\label{eq:estrutura-E1}
\end{equation}
thus obtaining (\ref{eq:est-psin1-(1)}).

Let us also deduce the equations (\ref{eq:est-psin1-(2)}).

An analogue procedure as above (starting by multiplying both hand
sides of (\ref{eq:teo1-aux1}) by
$\mathbb{E}_1P_{n}\,\mathbb{E}_2P_{n}$, using the properties
(\ref{eq:M-quoc}) and (\ref{eq:D-quoc-E2}), as well as the identity
$\mathbb{E}_2P_{n-1}^{(1)}\,\mathbb{E}_2P_{n-1}-\mathbb{E}_2P_{n}\,\mathbb{E}_2P_{n-2}^{(1)}=\prod_{k=0}^{n-1}\gamma_k
$) yields the equations
\begin{equation}
\begin{cases}
A\mathbb{D}P_{n}+C/2\,\mathbb{E}_1P_{n}+B\mathbb{E}_1P_{n-1}^{(1)}-\Theta_{n-1}\mathbb{E}_2P_{n-1}=(\tilde{l}_{n-1}+\Delta_y \tilde{\pi}_{n-1})\mathbb{E}_2P_{n}\,,\\
A\mathbb{D}P^{(1)}_{n-1}-C/2\,\mathbb{E}_1P^{(1)}_{n-1}-D\mathbb{E}_1P_{n}-\Theta_{n-1}\mathbb{E}_2P_{n-2}^{(1)}=(\tilde{l}_{n-1}+\Delta_y
\tilde{\pi}_{n-1})\mathbb{E}_2P_{n-1}^{(1)}\,,
\end{cases}\label{eq:estrutura-E2}
\end{equation}
with $\tilde{l}_{n-1}, \tilde{\pi}_{n-1}$ polynomials (in $x$).

 Let us now
prove that  $l_{n-1}=\tilde{l}_{n-1}$ and
$\pi_{n-1}=-\tilde{\pi}_{n-1}$.

 To obtain $\pi_{n-1}=-\tilde{\pi}_{n-1}$
 we multiply the first equation of
(\ref{eq:estrutura-E1}) by $\mathbb{E}_2P_{n}$, multiply the first
equation of (\ref{eq:estrutura-E2}) by $\mathbb{E}_1P_{n}$, and add
the resulting equations, thus obtaining
\begin{multline*}A\mathbb{D}P_{n}(\mathbb{E}_1P_{n}+\mathbb{E}_2P_{n})+\frac{C}{2}\left((\mathbb{E}_1P_{n})^2+(\mathbb{E}_2P_{n})^2\right)
+B\left(\mathbb{E}_1(P_{n}P^{(1)}_{n-1})+\mathbb{E}_2(P_{n}P^{(1)}_{n-1})\right)\\-\Theta_{n-1}\left(\mathbb{E}_1P_{n-1}\,\mathbb{E}_2P_{n}
+\mathbb{E}_2P_{n-1}\,\mathbb{E}_1P_{n}\right)=({l}_{n-1}+\tilde{l}_{n-1}+\Delta_y
(\pi_{n-1}+\tilde{\pi}_{n-1}))\mathbb{E}_1P_{n}\,\mathbb{E}_2P_{n}\,.\end{multline*}
Taking into account the Lemma \ref{eq:lemma-ident-aux}, the left
hand side of the above equation,  as well as
$\mathbb{E}_1P_{n}\,\mathbb{E}_2P_{n}$, are polynomials (in $x$).
Consequently, there must hold that
${l}_{n-1}+\tilde{l}_{n-1}+\Delta_y(\pi_{n-1}+\tilde{\pi}_{n-1})$ is
a polynomial in $x$. As $\Delta_y=2\sqrt{r}$, then
$\pi_{n-1}+\tilde{\pi}_{n-1}=0$, that is, $
\pi_{n-1}=-\tilde{\pi}_{n-1}.$

To obtain $l_{n-1}=\tilde{l}_{n-1}$ we interchange $y_1$ with $y_2$
in the first equation of (\ref{eq:estrutura-E1}), thus obtaining
\begin{equation}
A\mathbb{D}P_{n}+C/2\,\mathbb{E}_1P_{n}+B\mathbb{E}_1P_{n-1}^{(1)}-\Theta_{n-1}\mathbb{E}_2P_{n-1}=(l_{n-1}-\Delta_y
\pi_{n-1})\mathbb{E}_2P_{n}\,. \label{eq:auxproof-teo2}
\end{equation}
The comparison between (\ref{eq:auxproof-teo2}) and the first
equation of (\ref{eq:estrutura-E2}), where we use
$\tilde{\pi}_{n-1}=-{\pi}_{n-1}$, gives us
$l_{n-1}=\tilde{l}_{n-1}$. Hence, we obtain
(\ref{eq:est-psin1-(2)}).

 \vspace{0.2cm}

\noindent{\bf{Proof\; of\;$(b) \Rightarrow (a)$.}}\\
 Let us define
$\psi_n=\left[
\begin{matrix}
  P_{n+1} \\
  P_n^{(1)} \\
\end{matrix}
\right].$ From (\ref{eq:ttrr-Pn}) and (\ref{eq:ttrr-Pn1})
 there holds
\begin{equation}\psi_n=(x-\beta_n)\psi_{n-1}-\gamma_n\psi_{n-2}\,, \quad n\geq
1\,,\;\; \psi_{-1}=\begin{bmatrix}1 \\ 0 \end{bmatrix},\;\;
\psi_{0}=\begin{bmatrix}x-\beta_0 \\ 1 \end{bmatrix}\,.
\label{eq:rrpsi}\end{equation}

With the notation
$\mathbb{D}\psi_n=\begin{bmatrix}\mathbb{D}P_{n+1}\\
\mathbb{D}P^{(1)}_{n}
\end{bmatrix}, \; \mathbb{E}_j\psi_n=\begin{bmatrix}\mathbb{E}_jP_{n+1}\\
\mathbb{E}_jP^{(1)}_{n}
\end{bmatrix},$ $j=1,2,$   (\ref{eq:est-psin1-(1)}) reads as
\begin{equation}
A\mathbb{D}\psi_{n-1}=(l_{n-1}+\Delta_y
\pi_{n-1})\mathbb{E}_1\psi_{n-1}+\mathcal{C}\,\mathbb{E}_2\psi_{n-1}+\Theta_{n-1}\mathbb{E}_1\psi_{n-2}\,,\;\;
\label{eq:est1-vector}
\end{equation}
where $\mathcal{C}=\begin{bmatrix}-C/2 & -B\\ D & C/2\end{bmatrix}$.
In turn,  (\ref{eq:est1-vector}) reads as
\begin{equation*}
A\left(\frac{\psi_{n-1}(y_2)-\psi_{n-1}(y_1)}{y_2-y_1}\right)=(l_{n-1}+\Delta_y
\pi_{n-1})\psi_{n-1}(y_1)+\mathcal{C}\psi_{n-1}(y_2)+\Theta_{n-1}\psi_{n-2}(y_1)\,,
\end{equation*}
that is,
\begin{equation}
{A}_{n}\psi_{n-1}(y_1)+\mathcal{B}\psi_{n-1}(y_2)={C}_{n}\psi_{n-2}(y_1)\label{eq:aux-teo1-n}
\end{equation}
with the functions $A_n, C_n$ given by
$A_n=-\displaystyle\frac{A}{y_2-y_1}-l_{n-1}-\Delta_y \pi_{n-1},
{C}_{n}=\Theta_{n-1}\,,$ and the matrix $\mathcal{B}$ given by
$\mathcal{B}=\displaystyle\frac{A}{y_2-y_1}I-\mathcal{C}$,  $I$
denoting the identity matrix of order $2.$ Note that $A_n, C_n$, as
well as the entries of $\mathcal{B}$, are rational functions of
fixed
 degrees (independent of $n$) of the variables $x$ and $y_1(x)$ (indeed, $y_2(x)$ can be replaced by $-2(\hat{b}x+\hat{d})/\hat{a}-y_1(x)$).

 Taking $n+1$ in (\ref{eq:aux-teo1-n}) and using the
recurrence relation (\ref{eq:rrpsi}) we get
\begin{equation}
\tilde{{A}}_{n}\psi_{n-1}(y_1)+\tilde{\mathcal{B}}_{n}\psi_{n-1}(y_2)
=\tilde{{C}}_{n}\psi_{n-2}(y_1)+\tilde{\mathcal{D}}_{n}\psi_{n-2}(y_2)\,,\label{eq:aux-teo1-n+1}
\end{equation}
with
\begin{equation*}
\tilde{{A}}_{n}=(y_1-\beta_n){A}_{n+1}-{C}_{n+1},\;
\tilde{\mathcal{B}}_{n}=(y_2-\beta_n)\mathcal{B},\;
\tilde{{C}}_{n}=\gamma_n{A}_{n+1},\;\tilde{\mathcal{D}}_{n}=\gamma_n\mathcal{B}\,.
\end{equation*}

Now, we gather (\ref{eq:aux-teo1-n}) and (\ref{eq:aux-teo1-n+1}) in
the system
\begin{equation}\mathcal{E}_n\begin{bmatrix}\psi_{n-1}(y_1)\\
\psi_{n-1}(y_2)\end{bmatrix}=\mathcal{F}_n\begin{bmatrix}\psi_{n-2}(y_1)\\
\psi_{n-2}(y_2)\end{bmatrix}\,,\label{eq:teo1-aux2}\end{equation}
where $\mathcal{E}_n$ and $\mathcal{F}_n$ are the block matrices
\begin{equation*}
\mathcal{E}_n=\begin{bmatrix}{A}_nI & \mathcal{B}\\
\tilde{{A}}_nI & \tilde{\mathcal{B}}_n\end{bmatrix}\,,\;\; \mathcal{F}_n=\begin{bmatrix}{C}_nI & 0_{2\times 2}\\
\tilde{{C}}_nI & \tilde{\mathcal{D}}_n\end{bmatrix}\,.
\end{equation*}
Note that $\mathcal{E}_n$ is invertible,
\begin{equation}\mathcal{E}_n^{-1}=\displaystyle\frac{1}{\zeta_n}\begin{bmatrix}(y_2-\beta_n)I & -I\\
-\tilde{A}_n\mathcal{B}^{-1} & {A}_n\mathcal{B}^{-1}
\end{bmatrix}\,,\;\; \zeta_n=-(y_1-\beta_n)A_{n+1}+(y_2-\beta_n)A_n+C_{n+1}\,.\label{eq:inv-calE}\end{equation}

 From (\ref{eq:teo1-aux2}) there follows
\begin{equation}\begin{bmatrix}\psi_{n-1}(y_1)\\
\psi_{n-1}(y_2)\end{bmatrix}=\mathcal{G}_n\begin{bmatrix}\psi_{n-2}(y_1)\\
\psi_{n-2}(y_2)\end{bmatrix}\,, \;\;
\mathcal{G}_n=\mathcal{E}^{-1}_n\mathcal{F}_n\,,\label{eq:teobimpla-aux3}\end{equation}
being $\mathcal{G}_n$ an invertible matrix,
 as it is a product of
invertible matrices.

 Take $n+1$ in
(\ref{eq:teobimpla-aux3}).
 On the one hand we have
\begin{equation}
\begin{bmatrix}\psi_{n}(y_1)\\
\psi_{n}(y_2)\end{bmatrix}=
\mathcal{G}_{n+1}\begin{bmatrix}\psi_{n-1}(y_1)\\
\psi_{n-1}(y_2)\end{bmatrix}\label{eq:onehand}
\end{equation}
and, on the other hand, using the three term recurrence relation
(\ref{eq:rrpsi}), we have
\begin{equation}
\begin{bmatrix}\psi_{n}(y_1)\\
\psi_{n}(y_2)\end{bmatrix}=\begin{bmatrix}(y_1-\beta_n)I & 0_{2\times 2}\\
0_{2\times 2} & (y_2-\beta_n)I\end{bmatrix}\begin{bmatrix}\psi_{n-1}(y_1)\\
\psi_{n-1}(y_2)\end{bmatrix}-\gamma_n\begin{bmatrix}\psi_{n-2}(y_1)\\
\psi_{n-2}(y_2)\end{bmatrix}\,.\nonumber
\end{equation}
Using (\ref{eq:teobimpla-aux3}) in the above equation, we get
\begin{equation}
\begin{bmatrix}\psi_{n}(y_1)\\
\psi_{n}(y_2)\end{bmatrix}=\left(\begin{bmatrix}(y_1-\beta_n)I & 0_{2\times 2}\\
0_{2\times 2} & (y_2-\beta_n)I\end{bmatrix}-\gamma_n
\mathcal{G}_n^{-1} \right)\begin{bmatrix}\psi_{n-1}(y_1)\\
\psi_{n-1}(y_2)\end{bmatrix}\,.\label{eq:otherhand}
\end{equation}
Consequently, (\ref{eq:onehand}) and (\ref{eq:otherhand}) yield
\begin{equation}
\mathcal{G}_{n+1}=\begin{bmatrix}(y_1-\beta_n)I & 0_{2\times 2}\\
0_{2\times 2} &
(y_2-\beta_n)I\end{bmatrix}-\gamma_n\mathcal{G}^{-1}_{n}\,.\label{eq:Gn-Gn+1}
\end{equation}

Let us compute $\mathcal{G}_{n}$ as well as $\mathcal{G}^{-1}_{n}$.
From  (\ref{eq:inv-calE}), we obtain
\begin{equation}
\mathcal{G}_{n}=\begin{bmatrix}X_nI & Y_n\mathcal{B}\\
U_n\mathcal{B}^{-1} & V_n I\end{bmatrix}\,,\label{eq:Gn}
\end{equation}
where $X_n, Y_n, U_n, V_n$ are given by
\begin{multline*}
X_n=\left((y_2-\beta_n)C_{n}-\gamma_nA_{n+1}\right)/\zeta_n,
\;Y_n=-\gamma_n/\zeta_n, \; U_n=
(A_n\tilde{C}_n-\tilde{A}_nC_n)/\zeta_n\,,\\ V_n=\gamma_n
A_n/\zeta_n\,.
\end{multline*}
Note that $X_n, Y_n, U_n, V_n$ are again rational functions of fixed
degrees of $x$ and $y_1(x)$.

From (\ref{eq:Gn}) there follows
\begin{equation}
\mathcal{G}^{-1}_{n}=\frac{1}{\delta_n}\begin{bmatrix}V_nI & -Y_n\mathcal{B}\\
-U_n\mathcal{B}^{-1} & X_n I\end{bmatrix}, \quad
\delta_n=X_nV_n-Y_nU_n\,.\label{eq:Gn-inversa}
\end{equation}
 Taking into account (\ref{eq:Gn}) and (\ref{eq:Gn-inversa}), (\ref{eq:Gn-Gn+1})
 yields
\begin{eqnarray}
&&X_{n+1}=(y_1-\beta_n)-\gamma_nV_n/\delta_n\,,\label{eq:rel-X}\\
&&Y_{n+1}=\gamma_nY_n/\delta_n\,, \label{eq:rel-Y}\\
&&U_{n+1}=\gamma_n U_n/\delta_n\, \label{eq:rel-U}\\
&&V_{n+1}=(y_2-\beta_n)-\gamma_n X_n/\delta_n\,. \label{eq:rel-V}
\end{eqnarray}
From (\ref{eq:rel-X})-(\ref{eq:rel-V}) there follows that
$\delta_{n+1}$ (recall $\delta_{n+1}
={X}_{n+1}{V}_{n+1}-{Y}_{n+1}{U}_{n+1}$) satisfies
\begin{equation}
\delta_{n+1}=(y_1-\beta_n)(y_2-\beta_n)-\gamma_n\left((y_1-\beta_n){X}_n+(y_2-\beta_{n}){V}_n\right)\frac{1}{\delta_n}
+\frac{\gamma_n^2}{\delta_n}\,.\label{eq:bimpla-delta-n+1}
\end{equation}

 Now we proceed in analogy with
\cite[Lemma 5.1]{magnus-LHsnul} and
\cite[Theorem~1]{magnus-snul-sc}. Write $\delta_n=\mu_n/\mu_{n-1}$
and use such an expression in (\ref{eq:rel-X}), (\ref{eq:rel-V}) and
(\ref{eq:bimpla-delta-n+1}), thus obtaining
\begin{eqnarray*}
&&\mu_n{X}_{n+1}=(y_1-\beta_n)\mu_n-\gamma_n\mu_{n-1}{V}_n\,,\label{eq:rel-X-2}
\\
&&\mu_n{V}_{n+1}=(y_2-\beta_n)\mu_n-\gamma_n\mu_{n-1}{X}_n\,,\label{eq:rel-V-2}\\
&&\mu_{n+1}=(y_1-\beta_n)(y_2-\beta_n)\mu_n-\gamma_n\left((y_1-\beta_n){X}_n+(y_2-\beta_{n}){V}_n\right)\mu_{n-1}
+\gamma_n^2\mu_{n-1}\,. \hspace{0.6cm}
\end{eqnarray*}
The change of variables $\hat{{X}}_{n+1}=\mu_n{X}_{n+1},\;
 \hat{{V}}_{n+1}=\mu_n{V}_{n+1}\,$
yields the relations
\begin{eqnarray}
&&\hat{{X}}_{n+1}=(y_1-\beta_n)\mu_n-\gamma_n\hat{{V}}_n\,,\nonumber\label{eq:rel-X-3}
\\
&&\hat{{V}}_{n+1}=(y_2-\beta_n)\mu_n-\gamma_n\hat{{X}}_n\,, \nonumber\label{eq:rel-V-3}\\
&&\mu_{n+1}=(y_1-\beta_n)(y_2-\beta_n)\mu_n-\gamma_n\left((y_1-\beta_n)\hat{X}_n+(y_2-\beta_{n})\hat{V}_n\right)+\gamma_n^2\mu_{n-1}.
\hspace{0.6cm}\label{eq:rec-mu}
\end{eqnarray}
Let us emphasize that the above recurrence relations for $\hat{X}_n,
\hat{V}_n$ and $\mu_n$ are precisely the recurrence relations
satisfied by the products of solutions of the three term recurrence
relation  $\tau_{n+1}=(x-\beta_n)\tau_n-\gamma_n\tau_{n-1}$ at $y_1$
and $y_2$. Indeed, if
$$\xi_{n+1}=(y_1-\beta_n)\xi_n-\gamma_n\xi_{n-1}\,, \;\;\eta_{n+1}=(y_2-\beta_n)\eta_{n}-\gamma_n\eta_{n-1}\,,$$ then the above
recurrence relation for $\hat{X}_n, \hat{V}_n, \mu_n $ is precisely
the relation satisfied by the products $\xi_n \eta_{n-1}$,
$\xi_{n-1}\eta_n,$ and $ \xi_n \eta_n$, respectively. Taking into
account that a basis of solutions of the three term recurrence
relation $\tau_{n+1}=(x-\beta_n)\tau_n-\gamma_n\tau_{n-1}$ is
constituted by $\{P_n\}$ and $\{q_n\}$  (cf. (\ref{eq:ttrr-Pn}) and
(\ref{eq:ttrr-qn})), the following must hold: $\xi_n$ must be a
combination of $P_n(y_1)$ and $q_n(y_1)$, and $\eta_n$ must be a
combination of $P_n(y_2)$ and $q_n(y_2)$, that is,
$$\xi_n=\alpha_1P_n(y_1)+\lambda_1q_n(y_1)\,, \;\; \eta_n=\alpha_2P_n(y_2)+\lambda_2q_n(y_2)\,.$$

Hence, one has
\begin{equation}
\mu_n=\alpha\,P_{n}(y_1)P_n(y_2)+\beta\,P_n(y_1)q_n(y_2)+\gamma\,q_n(y_1)P_n(y_2)+\delta\,q_n(y_1)q_n(y_2)\,,\label{eq:mun-final}
\end{equation}
with  $\alpha, \beta, \gamma, \delta$ rational functions of $x$,
independent of $n$. Taking $n=0$ in (\ref{eq:mun-final}) we obtain
$$\mu_0=\alpha+\beta\,q_0(y_2)+\gamma\,q_0(y_1)+\delta\,q_0(y_1)q_0(y_2)\,,$$
and such a relation is \begin{equation}\nu_1 \mathbb{D}S =\nu_2
\mathbb{E}_1S\,\mathbb{E}_2S+\nu_3\mathbb{M}S+\nu_4\,,\label{eq:ric-nu}\end{equation}
where the $\nu$'s are given by
$$\nu_1=\frac{(\gamma-\beta)}{2}(y_2-y_1), \; \nu_2=\delta, \;
\nu_3=\gamma+\beta, \; \nu_4=\alpha-\mu_0\,.$$

Note that the $\nu$'s are rational functions  of fixed
 degrees of $x$ and $y_1(x)$.
 Therefore,
 the
 multiplication of (\ref{eq:ric-nu}) by a polynomial in $x$ and $y_1(x)$ of appropriate degree, and taking into account
$y_1=p-\sqrt{r}$, with $p, r$ polynomials, and $\Delta_y=2\sqrt{r}$
(cf. (\ref{eq:y1-y2})), gives us
\begin{equation}(A-\Delta_y\widehat{A}) \mathbb{D}S
=(B-\Delta_y\widehat{B})
\mathbb{E}_1S\,\mathbb{E}_2S+(C-\Delta_y\widehat{C})\mathbb{M}S+D-\Delta_y\widehat{D}\,,\label{eq:ric-y1}\end{equation}
where  $A, B, C, D, \widehat{A}, \widehat{B},\widehat{C},
\widehat{D}$ are polynomials in $x$. On the other hand, there also
holds that the $\nu$'s are rational functions of $x$ and $y_2(x)$
(indeed, one can replace $y_1(x)$ by
$-2(\hat{b}x+\hat{d})/\hat{a}-y_2(x)$). Taking into account that
$y_2=p+\sqrt{r}$, one also obtains
\begin{equation}(A+\Delta_y\widehat{A}) \mathbb{D}S
=(B+\Delta_y\widehat{B})
\mathbb{E}_1S\,\mathbb{E}_2S+(C+\Delta_y\widehat{C})\mathbb{M}S+D+\Delta_y\widehat{D}\,.\label{eq:ric-y2}\end{equation}
Finally, the sum of (\ref{eq:ric-y1}) with (\ref{eq:ric-y2}) yields
the required Riccati equation for $S$ with polynomial coefficients.

 \vspace{0.2cm}

\noindent{\bf{Proof of $(a) \Rightarrow (c)$.}}\\
Note that as $q_{-1}=1, q_0=S,$ then $A \mathbb{D}S =B
\mathbb{E}_1S\,\mathbb{E}_2S+C\mathbb{M}S+D$ is $$A \mathbb{D}q_{n}
   =(l_{n-1}+\Delta \pi_{n-1})\mathbb{E}_1 q_{n}+\left(B\mathbb{E}_1S+C/2\right)\mathbb{E}_2q_{n}+
  \Theta_{n-1}\,\mathbb{E}_1q_{n-1}\,$$ to $n=0$, with  $l_{-1}=C/2, \pi_{-1}=0, \Theta_{-1}=D.$

Let us now deduce the above difference equation for $n\geq 1.$

 Applying $A\mathbb{D}$ to $q_{n}=P_{n}S-P^{(1)}_{n-1},\;
n\geq 1 $ (cf. (\ref{eq:defi-qn})), and using the property
(\ref{eq:Dprod-gf-E1}) we obtain
\begin{equation*}
A\mathbb{D}q_{n}=A\mathbb{D}P_{n}\,\mathbb{E}_1S+A\mathbb{D}S\,\mathbb{E}_2P_{n}-A\mathbb{D}P_{n-1}^{(1)}\,.\label{eq:teo2aux1}\nonumber
\end{equation*}
Using (\ref{eq:est-psin1-(1)})
 as well as
(\ref{eq:defi-Ric-S}) in the above equation, we obtain
\begin{multline*}
A\mathbb{D}q_{n}=(l_{n-1}+\Delta_y
\pi_{n-1})\mathbb{E}_1(P_{n}S-P^{(1)}_{n-1})+\left(B
\mathbb{E}_1S+C/2\,\right)\mathbb{E}_2(P_{n}S-P_{n-1}^{(1)})\\+\Theta_{n-1}\mathbb{E}_1(P_{n-1}S-P_{n-2}^{(1)})\,,
\end{multline*}
thus (\ref{eq:est-qn-(1)}) follows.

Note that an analogue procedure as above (starting by the use of the
property (\ref{eq:Dprod-gf-E2}) combined with the equations
(\ref{eq:est-psin1-(2)})
 as well as
(\ref{eq:defi-Ric-S})) yields (\ref{eq:est-qn-(2)}).

\vspace{0.2cm}

\noindent{\bf{Proof of $(c) \Rightarrow (a)$.}}\\
Take $n=0$ in (\ref{eq:est-qn-(1)}) and use the initial conditions
$l_{-1}=C/2, \pi_{-1}=0, \Theta_{-1}=D,$ thus getting the Riccati
equation for $S$, (\ref{eq:defi-Ric-S}).

Note that, in the same manner, taking $n=0$ in (\ref{eq:est-qn-(2)})
we get  (\ref{eq:defi-Ric-S}).

\section*{Acknowledgements}

This work was partially supported by Centro de Matem\'atica da
Universidade de Coimbra (CMUC), funded by the European Regional
Development Fund through the program COMPETE and by the Portuguese
Government through the FCT - Funda\c{c}\~ao para a Ci\^encia e a
Tecnologia under the project PEst-C/MAT/UI0324/2011.

\end{document}